\documentclass[11pt]{article}
\usepackage{psfrag,graphicx,fullpage}
\usepackage{algorithmic,algorithm,amsmath,pdfsync}

\usepackage[round]{natbib}

\title{Convex Relaxations for Subset Selection}

\author{Francis Bach\thanks{INRIA-WILLOW Project-Team, Laboratoire d'Informatique de l'Ecole Normale Sup\'erieure (CNRS/ENS/INRIA UMR 8548), 23, avenue d'Italie, 75214 Paris, France. \texttt{francis.bach@mines.org}}, 
Selin Damla Ahipa{\c s}ao{\u g}lu\thanks{ORFE, Princeton University, Princeton, NJ 08544. \texttt{sahipasa@princeton.edu}}, Alexandre d'Aspremont\thanks{ORFE, Princeton University, Princeton, NJ 08544. \texttt{aspremon@princeton.edu}}}

\newtheorem{theorem}{theorem}
\newtheorem{proposition}[theorem]{Proposition}
\newcommand{\QED}{~~\rule[-1pt]{6pt}{6pt}}
\newenvironment{proof}{\textbf{Proof.}}{\QED}

\newcommand{\BEAS}{\begin{eqnarray*}}
\newcommand{\EEAS}{\end{eqnarray*}}
\newcommand{\BEA}{\begin{eqnarray}}
\newcommand{\EEA}{\end{eqnarray}}
\newcommand{\BEQ}{\begin{equation}}
\newcommand{\EEQ}{\end{equation}}
\newcommand{\BIT}{\begin{itemize}}
\newcommand{\EIT}{\end{itemize}}
\newcommand{\BNUM}{\begin{enumerate}}
\newcommand{\ENUM}{\end{enumerate}}
\newcommand{\BMI}{\begin{minipage}}
\newcommand{\EMI}{\end{minipage}}

\newcommand{\BA}{\begin{array}}
\newcommand{\EA}{\end{array}}
\newcommand{\BC}{\begin{center}}
\newcommand{\EC}{\end{center}}

\newcommand{\ones}{\mathbf 1}

\newcommand{\reals}{{\mbox{\bf R}}}

\newcommand{\symm}{{\mbox{\bf S}}}  

\newcommand{\Tr}{\mathop{\bf Tr}}
\newcommand{\diag}{\mathop{\bf diag}}
\newcommand{\lambdamax}{{\lambda_{\rm max}}}

\newcommand{\Card}{\mathop{\bf Card}}
\newcommand{\idm}{{\bf I}}

\newcommand{\Expect}{\mathop{\bf E{}}}
\newcommand{\Prob}{\mathop{\bf Prob}}

\newcommand{\argmax}{\mathop{\rm argmax}}

\begin{document}
\maketitle

\begin{abstract}
We use convex relaxation techniques to produce lower bounds on the optimal value of subset selection problems and generate good approximate solutions. We then explicitly bound the quality of these relaxations by studying the approximation ratio of sparse eigenvalue relaxations. Our results are used to improve the performance of branch-and-bound algorithms to produce exact solutions to subset selection problems.
\end{abstract}

\section{Introduction}
We focus here on the subset selection problem., i.e., solving least squares regressions while constraining the number of nonzero regression variables to be less than a certain target. This problem is often called feature selection or sparse least-squares. Its combinatorial nature makes subset selection intractable. Several techniques have been derived to produce good approximate solutions however, using for example greedy algorithms or sparsity inducing penalties.

Given a design matrix $X\in\reals^{n \times p}$ and a response vector $y\in\reals^n$, we consider the following subset selection problem
\BEQ\label{eq:subsel}
\BA{ll}
\mbox{minimize} & \|y-Xw\|_2^2\\
\mbox{subject to} & \Card(w)\leq k,
\EA\EEQ
in the variable $w\in\reals^p$, where $k$ is a parameter controlling sparsity. 
It was shown in \cite{Nata95} that while~\eqref{eq:subsel} is NP-Hard, simple greedy algorithms can efficiently produce good approximate solutions. Subset selection can also be understood as $\ell_0$ norm constrained regression (or approximation) and a very large body of works focused on replacing the combinatorial $\ell_0$ norm with a convex $\ell_1$ norm constraint, with $\ell_1$ norm regression usually known as LASSO \cite{Tibs96}. Explicit variable selection consistency results have been derived in certain regimes (see e.g. \cite{Mein07a}), and recent results \cite{Dono05a,Cand07} have shown that under certain conditions on the design matrix $X$, the solutions of the $\ell_1$ problem coincided with that of the $\ell_0$ problem. Several authors have attacked the $\ell_0$ problem directly, with \cite{Nare77,Hand81,Furn00,Mogh08} using branch-and-bound techniques to produce exact solutions to problem~\eqref{eq:subsel}, with \cite{Mogh08} in particular using interlacing properties of eigenvalues to speedup branch-and-bound methods. Solving the $\ell_0$ problem in~\eqref{eq:subsel} even for small values of $p$ has direct applications in image denoising \cite{Elad06,Mair08}. 

All the algorithms listed above produce good approximate solutions, hence {\em upper bounds} on the optimal value of the subset selection problem~\eqref{eq:subsel}. Our first contribution here is to use convex relaxation techniques to produce {\em lower bounds} on the optimal value of~\eqref{eq:subsel}. In particular, this result allows us to bound the suboptimality of approximate solutions and improve the performance of branch-and-bound algorithms for subset selection. We also use randomization techniques to generate good solutions to~\eqref{eq:subsel}, often improving on solutions produced by greedy or LASSO algorithms. Our next main contribution is to derive approximation bounds on the performance of the sparse eigenvalue relaxation/randomization algorithm. Finally, we test our algorithms on various subset selection problems and show that the lower bound derived here considerably reduces the number of branches required to produce an optimal solution to~\eqref{eq:subsel}. 

The paper is organized as follows. In Section~\ref{s:relax} we show how to produce lower bounds on the optimal value of problem~\eqref{eq:subsel} using relaxation bounds on sparse eigenvalues. In Section~\ref{s:sols}, we describe greedy, randomization and branch-and-bound  algorithms to generate good approximate solutions $w$ to~\eqref{eq:subsel} using the product of the relaxation. In Section~\ref{s:tight} we produce a bound on the approximation ratio of sparse eigenvalues, thus bounding the quality of the approximation of the subset selection bounds derived in Section~\ref{s:relax}. Section~\ref{s:algos} shows how to efficiently solve our semidefinite relaxation using first-order methods. Finally, Section~\ref{s:numres} presents some numerical experiments.

\subsection*{Notations}
Given matrices $X,Y\in\reals^{n\times p}$, we write $X\circ Y$ their Schur (componentwise) product, while $\lambdamax(X)$ is the leading eigenvalue of $X$, $\|X\|_1$ the sum of absolute values of the coefficients in $X$ and $X_i$ is the $i^th$ column of $X$. We let $\symm_p$ be the set of symmetric matrices, and for $Y\in\symm_p$, we write $\diag(Y)\in\reals^p$ its diagonal. When $y\in\reals^p$, $\diag(y)\in\symm_p$ denotes de diagonal matrix with diagonal coefficients equal to the coefficients of $y$, while $\Card(y)$ is the number of nonzero coefficients in~$y$.

\section{Relaxation \& Lower Bounds}\label{s:relax}
Following \cite{dasp08b} for example, we first recall how solving problem~\eqref{eq:subsel} is equivalent to computing sparse eigenvalues of a matrix formed using $X$ and $y$. We let $\psi(k)$ be the optimal value of the subset selection problem~\eqref{eq:subsel}, with
\BEQ\label{eq:psik}
\BA{rll}
\psi(k)=&\mbox{min.} & \|y-Xw\|_2^2\\
&\mbox{s.t.} & \Card(w)\leq k,
\EA\EEQ
in the variable $w\in\reals^p$, where $k$ is again a parameter controlling sparsity. We can rewrite this
\BEAS
\psi(k)&=&\min_{\substack{\ones^Tu\leq k\\u\in\{0,1\}^p}}~\min_{\substack{\Card(w)\leq k\\ w\in\scriptsize\reals^p}} ~\|y-X\diag(u)w\|_2^2\\
&=&\min_{\substack{\ones^Tu\leq k\\u\in\{0,1\}^p}}~\min_{w\in\scriptsize\reals^p} ~\|y-X\diag(u)w\|_2^2\\
&=&\min_{\substack{\ones^Tu\leq k\\u\in\{0,1\}^p}}~\min_{\|w\|_2=1}~\min_{\nu\in\scriptsize\reals} ~ \|y-X\diag(u)\nu w\|_2^2,
\EEAS
which, after minimizing explicitly over $\nu$, becomes
\[
\min_{\substack{\ones^Tu\leq k\\u\in\{0,1\}^p}}~\min_{\|w\|_2=1} ~y^Ty-\frac{(y^TX\diag(u)w)^2}{w^T\diag(u)X^TX\diag(u)w}.
\]
This means that $\psi(k)\leq y^Ty - \rho$ if and only if
\[
\max_{\substack{\ones^Tu\leq k\\u\in\{0,1\}^p}} ~\max_{\|w\|_2=1} ~ \frac{(y^TX\diag(u)w)^2}{w^T\diag(u)X^TX\diag(u)w} \leq \rho.
\]
We can rewrite this condition
\[
wX^T(yy^T- \rho\idm)Xw \leq 0, \quad \mbox{when }\|w\|_2\leq 1,\,\Card(w)\leq k
\]
which is equivalent to
\BEQ\label{eq:keig-cond}
\lambda^k_\mathrm{max}(X^T(yy^T- \rho\idm)X) \leq 0.
\EEQ
Here, $\lambda^k_\mathrm{max}(A)$ is the $k$ sparse maximum eigenvalue of a matrix $A\in\symm_p$, defined as
\BEQ\label{eq:sparse-eig}
\BA{rll}
\lambda^k_\mathrm{max}(A) =& \mbox{max.} & x^TAx\\
& \mbox{s. t.} & \|x\|= 1,\,\Card(x)\leq k
\EA\EEQ
in the variable $x\in\reals^p$. Relaxation bounds for sparse eigenvalues $\lambda^k_\mathrm{max}(A)$ were derived in \cite{dasp04a,dasp08b}, with the bound in \cite{dasp04a} written
\BEQ\label{eq:relax}
\BA{rll}
\lambda^k_\mathrm{max}(A)\leq&\mbox{maximize} & \Tr AZ\\
&\mbox{subject to} & \|Z\|_1 \leq k\\
&& \Tr(Z)=1,\,Z\succeq 0
\EA\EEQ
in the variable $Z\in\symm_p$. We can summarize the above derivation in the following proposition. 

\begin{proposition}\label{prop:k-eig}
Given a design matrix $X\in\reals^{n \times p}$ and a response vector $y\in\reals^n$, consider the following subset selection problem 
\[\BA{rll}
\psi(k)=&\mbox{min.} & \|y-Xw\|_2^2\\
&\mbox{s.t.} & \Card(w)\leq k,
\EA\]
then 
\[
\psi(k)\geq y^Ty-\rho \quad\mbox{if and only if}\quad \lambda^k_\mathrm{max}(X^T(yy^T- \rho\idm)X) \leq 0,
\]
where $\lambda^k_\mathrm{max}(\cdot)$ is the sparse maximum eigenvalue function defined in~\eqref{eq:sparse-eig}.
\end{proposition}

\section{Approximate Solutions}\label{s:sols}
The relaxation detailed in~\eqref{eq:relax} produces a {\em lower bound} on the objective value of problem~\eqref{eq:subsel}. In this section, we describe how to use the solution of this relaxation to produce good approximate solution vectors $w$ to problem~\eqref{eq:subsel}, hence produce {\em upper bounds} on the solution value. We first describe greedy algorithms which can be used to solve problem~\eqref{eq:subsel} independently, or to improve solutions extracted from convex relaxations.

\subsection{Greedy methods}\label{ss:greedy}
To simplify notations here, we first define the following function, which computes the solution value of problem~\eqref{eq:subsel} {\em given} the support of the solution vector $w\in\reals^p$. Let $I\subset[1,p]$ be a index subset such that $w_i= 0$ if $i\notin I$, we write $I^c$ its complement in $[1,p]$ and let
\BEQ\label{eq:Ifval}
\mu(I)= \min_{w_{I^c}=0}\|y-Xw\|_2^2
\EEQ
in the variable $w\in\reals^p$. Note that while computing the optimal value of problem~\eqref{eq:subsel} is NP-Hard, computing $\mu(I)$ in~\eqref{eq:Ifval} is equivalent to forming a QR decomposition of the matrix $X_I\in\reals^{n\times k}$ where~$k$ is the cardinality of the support set $I$.

We can greedily construct approximate solutions to~\eqref{eq:Ifval} by scanning variables at each iteration to increase (or decrease) the size of the support as in the {\em forward greedy} algorithm is detailed in Algorithm~\ref{alg:greedy-search}. The {\em backward greedy algorithm} is similar but starts from the full support $[1,p]$ and progressively removes points.

\begin{algorithm}[h]
\caption{Forward Greedy Algorithm.} 
\label{alg:greedy-search} 
\begin{algorithmic} [1]
\REQUIRE $X \in \reals^{n \times p},~y\in\reals^n$, target cardinality $k^\mathrm{target}$.
\STATE Initialization: $I_0 = \emptyset$.
\FOR{$i=1$ to $k^\mathrm{target}$} 
\STATE Compute $i_k = \argmax_{i \notin I_{k-1}}  \mu(I_{k-1} \cup \{ i \})$
\STATE Set $I_{k}=I_{k-1}\cup\{i_k\}$ and compute $w_{k}$ as the minimizer of $\mu(I_k)$ in~\eqref{eq:Ifval}.
\ENDFOR 
\ENSURE Support sets $I_k$ for $w$ in problem~\eqref{eq:subsel}, with $k=1,\ldots,k^\mathrm{target}$.
\end{algorithmic} 
\end{algorithm} 

\subsection{Randomization}\label{ss:rand}
As in the MAXCUT relaxation by \cite{goem95} for example, we can use the matrix solution to the relaxation in~\eqref{eq:relax} to generate good approximate solutions to problem~\eqref{eq:subsel}. The solution matrix $Z$ in~\eqref{eq:relax} can be understood as a covariance matrix, and we use it to generate Gaussian vectors $z\sim{\cal N}(0,Z)$. The $k$ indices corresponding to the $k$ largest magnitude coefficients of the sample vectors $z$ then provide support sets  $I$ corresponding to nonzero coefficients in $w$. Given these support sets, one then solves for $\mu(I)$ in~\eqref{eq:Ifval} to get upper bounds on the optimal value of~\eqref{eq:subsel} and approximate solution vectors $w$. 

In the next section, we will also consider another much simpler randomization procedure whose performance can be completely characterized. This second procedure does not require solving relaxation~\eqref{eq:relax}, but simply computing a leading eigenvector $x$ of the matrix $A$ in~\eqref{eq:sparse-eig}.
Good approximate solutions $z\in\reals^p$ to problem~\eqref{eq:sparse-eig} are then randomly sampled 
with $z_i=1/\sqrt{k}$ with probability $p_i=k |x_i|/\|x\|_1$ and $z_i=0$ otherwise. We then prune $z$ using a few backward greedy step, whenever $\Card(z)>k$. While the complexity of this procedure is much lower than that of the full greedy algorithm, we will see in the next section that it produces solutions of comparable quality.

\subsection{Branch-and-bound algorithm}\label{s:bandb}
As in \cite{Furn00,Mogh08}, we can develop a branch-and-bound algorithm for finding optimal solutions to~\eqref{eq:subsel}. Suppose we are looking for a vector in $\reals^p$ with at most $k$ non-zero components, we need to enumerate at most $p \choose k$ subsets to find the best one. We start by dividing all possible subsets into two branches, one containing the first variable and one which does not. We further branch each of these branches into two, one containing the second variable and one not, etc.  At each node of the search tree, we have a subproblem that excludes certain variables (depending on branching decisions made so far). For each subproblem, we generate lower bounds using Proposition~\ref{prop:k-eig} by solving relaxation~\eqref{eq:relax}, and upper bounds when there are exactly $k$ variables left on the branch. We also generate upper bounds by applying a combination of the greedy algorithms and randomization techniques described above to the solutions of the relaxed problems. Obviously, we fathom a node whose lower bound exceeds the best upper bound since the branches diverging from this node cannot contain a better solution than the best solution found so far.

\section{Tightness}\label{s:tight}
The sparse eigenvalue problem in~\eqref{eq:sparse-eig} is closely connected to the $k$-Dense-Subgraph problem described in \cite{Kort93,Feig01a,Feig01} for example. The $k$-Dense-Subgraph problem seeks to find a principal submatrix of $A$ of dimension $k$ with largest coefficient sum. This is written
\[
\max_{\ones^Tu\leq k} ~ u^TAu
\]
in the variable $u\in\{0,1\}^{p}$. On the other hand, the problem of computing a sparse maximum eigenvalue can be written 
\[
\lambda^k_{\mathrm{max}}(A) = \max_{\ones^Tu\leq k} ~\max_{\|x\|=1} ~ u^T(A\circ xx^T)u
\]
in the variables $x\in\reals^p$, $u\in\{0,1\}^p$. We thus observe that computing sparse eigenvalues means solving a $k$-Dense-Subgraph problem over the result of an inner eigenvalue problem in $x$. Below, we first recall an approximation result on the backward greedy algorithm used in \cite{Mogh08}, which applies to {\em positive semidefinite} matrices $A$.

\begin{proposition}\label{prop:greedy}
Let $A\in\symm_p$, with $A\succeq 0$ and $k>0$ and suppose $\diag(A) \geq 0$. We have
\BEQ\label{eq:tight-greedy}
\frac{k}{p} \lambda_{\mathrm{max}}(A)  \leq \lambda^k_{\mathrm{max}}(A)  \leq \lambda_{\mathrm{max}}(A) 
\EEQ
where $\lambda^k_{\mathrm{max}}(A)$ is the optimal value of problem~\eqref{eq:sparse-eig}.
\end{proposition}
\begin{proof} From \cite[\S4.3.14]{Horn85}, when $A\succeq 0$, we have
\[
\lambda^i_{\mathrm{max}}(A) \geq \frac{i}{i+1} \lambda^{i+1}_{\mathrm{max}}(A) 
\]
for any $i\in[1,p-1]$. A simple recursion then gives the desired result.
\end{proof}

When $A$ is not positive semidefinite, we can adapt results from \cite{Feig97} to show
\BEQ\label{eq:tight-greedy}
\frac{k(k-1)}{p(p-1)} \lambda_{\mathrm{max}}(A)  \leq \lambda^k_{\mathrm{max}}(A)  \leq \lambda_{\mathrm{max}}(A).
\EEQ

When the coefficients of $A$ are nonnegative, we can obtain approximation bounds for basic randomization techniques similar to those developed in \cite{Feig97} for the $k$-Dense-Subgraph problem. The approximation ratio in this case  also decreases as $k/p$, which shows that the randomization algorithm has a performance comparable to that of the backward greedy method, while being significantly cheaper on large scale problems.

\begin{proposition}\label{prop:sdp-nonneg}
Let $A\in\symm_p$ such that $A_{ij}\geq 0,\, i,j=1,\ldots,p$ and $k>1$. The sparse eigenvalue problem defined in~\eqref{eq:sparse-eig} was written
\[
\BA{rll}
\lambda^k_\mathrm{max}(A) =& \mbox{max.} & x^TAx\\
& \mbox{s. t.} & \|x\|_2\leq 1,\,\Card(x)\leq k
\EA\]
in the variable $x\in\reals^p$. We then have
\BEQ\label{eq:tight-SDP-nonneg}
\frac{k}{p}\, \mu(k,p) \, \lambdamax(A) \leq \lambda^k_\mathrm{max}(A) \leq SDP_k(A) \leq \lambdamax(A)
\EEQ
where 
\BEQ \label{eq:mukp}
\mu(k,p) = \left(1-\frac{2}{k^{1/3}}\right)\left(1-\frac{p^2}{k^2} e^{-p^{1/9}/3}\right) \xrightarrow[p \rightarrow \infty]{}1
\EEQ
whenever $k\geq p^{1/3}$, for $p$ sufficiently large, where $SDP_k(A)$ is the optimal value of~\eqref{eq:relax}.
\end{proposition}
\begin{proof}
To maintain the parallel with \cite{Feig97}, we write $Z=xx^T$, where $x$ is a leading eigenvector of $A$. The matrix $Z$ then satisfies $\Tr Z=1$ and $Z\succeq 0$. The upper bound in~\eqref{eq:tight-SDP-nonneg} follows directly from~\cite{dasp04a} and we focus here on the lower bound. We randomly sample vectors $z\in\reals^p$ such that
\[
z_i=
\left\{\BA{l}
1/\sqrt{k} \quad\mbox{with probability } p_i=k\sqrt{Z_{ii}}/S,\\
0 \quad\mbox{otherwise.}
\EA\right.
\]
where $S=\sum_{i=1}^p \sqrt{Z_{ii}}$. We then have 
\BEAS
\Expect[z^TAz] & = & \Tr(A\Expect[zz^T])\\
& = & \sum_{i,j=1}^p k A_{ij} \sqrt{Z_{ii}Z_{jj}}/S^2\\
& \geq & \sum_{i,j=1}^p k A_{ij} \sqrt{Z_{ii}Z_{jj}}/p\\
& \geq & \frac{ k}{p} \Tr(AZ)
\EEAS
where the first inequality uses $\Tr Z=1$ and the last (Cauchy) inequality follows from the fact that $Z\succeq 0$ and $A\geq 0$. Now, let $q=\Prob[z^TAz \leq \Expect[z^TAz]/\beta]$ for some $\beta \geq 1$, we have
\[
\Expect[z^TAz] \leq q \frac{\Expect[z^TAz]}{\beta} + (1-q) \frac{\ones^TA\ones}{k}
\]
which means
\[
q \leq 1-\frac{\beta-1}{\beta \ones^TA\ones/k\Expect[z^TAz]-1}.
\]
because $\ones^TA\ones/k\geq \Expect[z^TAz] $. Now, using Chernoff's inequality as in \cite[Lem. 4.1]{Feig97} produces
\[
\Prob\left[\Card(z) - \ones^Tp \geq t \ones^Tp \right]\leq e^{-\frac{t^2 \ones^Tp}{3}},
\]
so, as in \cite[Th. 4.1]{Feig97}, when $k \geq p^{1/3}$
\[
\Prob\left[\Card(z) \geq k\left(1+k^{-1/3}\right)\right] \leq e^{-p^{1/9}/3}.
\]
We have
\[
\frac{\ones^TA\ones}{k\Expect[z^TAz]}\leq \frac{p \ones^TA\ones}{k^2 \Tr AZ} \leq \frac{p^2 \lambda_{\mathrm{max}}(A)}{k^2 \Tr AZ}= \frac{p^2}{k^2}
\]
which follows from $A\geq 0$, $\Tr AZ=\lambdamax(A)$,  $\Tr Z=1$ with $Z\succeq 0$. When $k\geq p^{1/3}$ and $p$ is large enough so that $p^2 e^{-p^{1/9}/3}/k^2<1$, we can enforce
\[
\beta > \frac{1}{1-p^2 e^{-p^{1/9}/3}/k^2}
\]
and thus get
\[
\beta > \frac{e^{p^{1/9}/3}-1}{e^{p^{1/9}/3}-{\ones^TA\ones}/{k\Expect[z^TAz]}},
\]
which means, using the bound on $q$ derived above,
\[
1 -q \geq \frac{\beta-1}{\beta \ones^TA\ones/k\Expect[z^TAz]-1} > e^{-\frac{t^2 \ones^Tp}{3}},
\]
which, combined with the deviation bounds detailed above, yields
\[
\Prob[z^TAz \geq \Expect[z^TAz]/\beta] = 1-q > e^{-p^{1/9}/3} \geq \Prob\left[\Card(z) \geq k\left(1+k^{-1/3}\right)\right].
\]
This shows that by sampling enough points~$z$, we can generate a vector $z_0\in\reals^p$ such that
\[
z_0^TAz_0 \geq \frac{k}{\beta p} \Tr(AZ) 
\quad\mbox{and}\quad
\Card(z_0)\leq k\left(1+k^{-1/3}\right)
\]
If we remove at most $k^{2/3}$ variables from $z_0$ using the backward greedy algorithm described in the previous section,~\eqref{eq:tight-greedy} shows that we loose at most a factor
\[
\frac{k(k-1)}{(k+k^{2/3})(k+k^{2/3}-1)}=1-\frac{2}{k^{1/3}}+\frac{3}{k^{2/3}}+o\left(\frac{1}{k^{2/3}}\right)
\]
and, when $p$ is large enough, we obtain a point $z_k$ such that
\[
z_k^TAz_k \geq \frac{k}{p} \left(1-\frac{2}{k^{1/3}}\right) \frac{\Tr(AZ)}{\beta}, 
\quad 
\|z_k\|_2\leq 1
\quad\mbox{and}\quad
\Card(z_k)\leq k,
\]
which means that $z_k$ is a feasible point of problem~\eqref{eq:sparse-eig}, and yields the desired result.
\end{proof}

Note that the randomization procedure detailed in the proof above is simpler than the one we used in Section~\ref{ss:rand}, producing bounds on the performance of the later one is unfortunately much harder. We can directly extend this last proposition to problems where $A$ has negative coefficients, but the bound is not proportional in this case.

\begin{proposition}\label{prop:sdp}
Let $A\in\symm_p$ and $k>1$. We have
\BEQ\label{eq:tight-SDP}
\min\left\{0,\min_{i,j=1,\ldots,p}A_{ij}\right\}k+\frac{k}{p}\mu(k,p) \lambdamax(A) \leq \lambda^k_\mathrm{max}(A) \leq SDP_k(A) \leq \lambdamax(A),
\EEQ
where $SDP_k(A)$ is the optimal value of~\eqref{eq:relax} and $\mu(k,p)$ is defined in~\eqref{eq:mukp}, whenever $k\geq p^{1/3}$ and $p$ is sufficiently large.
\end{proposition}
\begin{proof}
The function $\lambda^k_\mathrm{max}(\cdot)$ defined in~\eqref{eq:sparse-eig} is convex as a pointwise maximum of affine functions. This implies
\[
\lambda^k_\mathrm{max}(A)\geq \lambda^k_\mathrm{max}(A-\min_{ij}A_{ij}~\ones\ones^T) + \min_{ij}A_{ij} (\ones^Tx)^2
\]
for some vector $x$ satisfying $\|x\|_2= 1$ and $\Card(x)\leq k$. The matrix $A-\min_{ij}A_{ij}~\ones\ones^T$ is nonnegative and Proposition~\ref{prop:sdp-nonneg} shows that 
\[
\lambda^k_\mathrm{max}(A-\min_{ij}A_{ij}~\ones\ones^T) \geq \frac{k}{p} \mu(k,p)\lambdamax(A-\min_{ij}A_{ij}~\ones\ones^T).
\]
We then get
\BEAS
&&\frac{k}{p}\mu(k,p)  \lambdamax(A-\min_{ij}A_{ij}~\ones\ones^T) +  \min_{ij}A_{ij} (\ones^Tx)^2\\
& \geq &  \frac{k}{p}\mu(k,p)  (\lambdamax(A) -\min_{ij}A_{ij}~(\ones^Ty)^2) + \min_{ij}A_{ij} (\ones^Tx)^2\\
& \geq & \frac{k}{p}\mu(k,p)  \lambdamax(A)  + \min_{ij}A_{ij} (\ones^Tx)^2
\EEAS
when $\min_{ij}A_{ij}<0$, which follows from the convexity of $\lambdamax(\cdot)$, where $y$ is a leading eigenvector of $A-\min_{ij}A_{ij}~\ones\ones^T$. We conclude using $(\ones^Tx)^2\leq\|x\|_1^2\leq \Card(x) \|x\|_2^2=k$.
\end{proof}

\section{Convex Minimization Algorithm}\label{s:algos}
The relaxation in Section \ref{s:relax} meant solving
\[\BA{ll}
\mbox{maximize} & \Tr MZ\\
\mbox{subject to} & \|Z\|_1 \leq k\\
& \Tr(Z)=1,\,Z\succeq 0
\EA\]
in the variable $Z\in\symm_p$, where $M\in\symm_p$ was formed as $M=X^T(yy^T-\rho \idm)X$. We compute the dual of this problem by first writing it in a saddle-point format.
\[
\min_{\lambda\geq 0} ~\max_{\substack{\Tr(Z)=1\\Z\succeq 0}} \Tr MZ + \lambda (k - \|Z\|_1)
\]
which is also
\[
\min_{\lambda\geq 0}~ \max_{\substack{\Tr(Z)=1\\Z\succeq 0}}~ \min_{\|Y\|_\infty\leq 1} \Tr Z(M+\lambda Y)  + k \lambda
\]
in the variables $Z,Y\in\symm_p$. We can rewrite this as
\[
\min_{Y\in\scriptsize\symm_p}~ \max_{\substack{\Tr(X)=1\\X\succeq 0}} \Tr X(M + Y)  + k \|Y\|_\infty
\]
which is equivalent to
\BEQ\label{eq:relax-dual}
\min~ \lambdamax(M+Y) + k\|Y\|_\infty\\
\EEQ
in the variable $Y\in\symm_p$. This is a maximum eigenvalue minimization problem and can be solved efficiently using for example smooth first-order algorithms as in \cite{Nest03a}. Given an a priori bound on suboptimality, the total complexity of obtaining a solution up to accuracy $\epsilon$ then grows as
\[
O\left( \frac{k n^3 \sqrt{\log n}}{\epsilon} \right).
\]
Given an approximate solution $Y\in\symm_p$ to the dual, we can reconstruct a corresponding primal solution $Z$ by first solving
\[
Z=\argmax_{\substack{\Tr(Z)=1\\Z\succeq 0}} ~\Tr Z(M+Y)
\]
and checking if $\|Z\|_1\leq k$ (this last condition will always be satisfied if $Y$ is optimal).

\section{Numerical Results}\label{s:numres}
Table \ref{t:bandb} presents numerical experiments using branch-and-bound on a set of small artificial problems. We generate normally distributed matrices $X\in\reals^{n \times p}$, a random sparse vector $w$ whose cardinality is at most $k$, and a righthand side vector $y$, which is equal to $Xw+\epsilon$, where $\epsilon \in \reals^n$ is noise. The last four columns are related to the performance of the B\&B algorithm: the first gives the smallest number of nodes visited by the algorithm, the second provides the average number of nodes visited over all instances, the third shows the number of nodes in the complete enumeration tree while the fourth lists the average speedup. These results suggest that the lower bound obtained in this paper is effective in fathoming a significant number of nodes in the search tree. Out of these 160 small test instances, the forward greedy algorithm found the optimal solution for 105 problems, whereas the randomization algorithm followed by a greedy improvement step (which will be referred as the enhanced randomization algorithm from now on) was able to find the optimal solution for 113 problems. Unfortunately, the authors of \cite{Mogh08} did not release a software package and the ``leaps and bounds'' package released by the authors of \cite{Furn00} does not output the number of nodes it visits so direct comparisons were not possible. 

\begin{table}[ht]
{\small
\begin{center}
\caption{Number of nodes visited by the branch-and-bound algorithm.}\label{t:bandb} \vspace{1ex}
\begin{tabular}{|l|l|l||c||c|c|c|c|}\hline
$p$   &  $ n$  & $k$ & No. instances & B\&B (Best) & B\&B (Average) & $p\choose k$ & Speedup (Avg.) \\ \hline
20  &   10  &   2   &   100 & 35   &  194 & 380 &  2 \\
30  &   15 &   3   &   50 & 330   & 4 799  & 24 360 & 5  \\
40  &   20 &   4   &   10 & 42 236   &   98 236 & 2 193 360 & 22 \\
50  &   25 &   4   &    2  & 71 552   &   96 734 & 5 527 200 & 57 \\
  \hline
\end{tabular}\end{center}}
\end{table}

On larger instances where $p=100$ and $n=50$, the cardinality of $w$ was set to 2 and 4. Figure~\ref{fig:ilst} plots lower bounds (Low. Bnd.) on~\eqref{eq:subsel} generated by solving relaxation~\eqref{eq:relax}, the coarse solution points (Primal) extracted from the matrix $Z$ solving~\eqref{eq:relax}, the solutions (Greedy) obtained by the forward greedy algorithm, the LARS algorithm \cite{Efro04}, and the enhanced solutions (Rand) obtained by applying the randomization algorithm detailed in Section~\ref{ss:rand} to the matrix $Z$ solving~\eqref{eq:relax}. We observe that around the true cardinality of $w$ used in generating the problem instances, the enhanced relaxation sometimes outperforms both the forward greedy algorithm and LARS and always performs at least as good as the best of these two methods. 

\begin{figure}[!h]
\begin{center}
\begin{tabular}{cc}
\psfrag{k}[t][b]{Cardinality $k$}
\psfrag{Psik}[b][t]{$\psi(k)$}
\includegraphics[width=.47\textwidth]{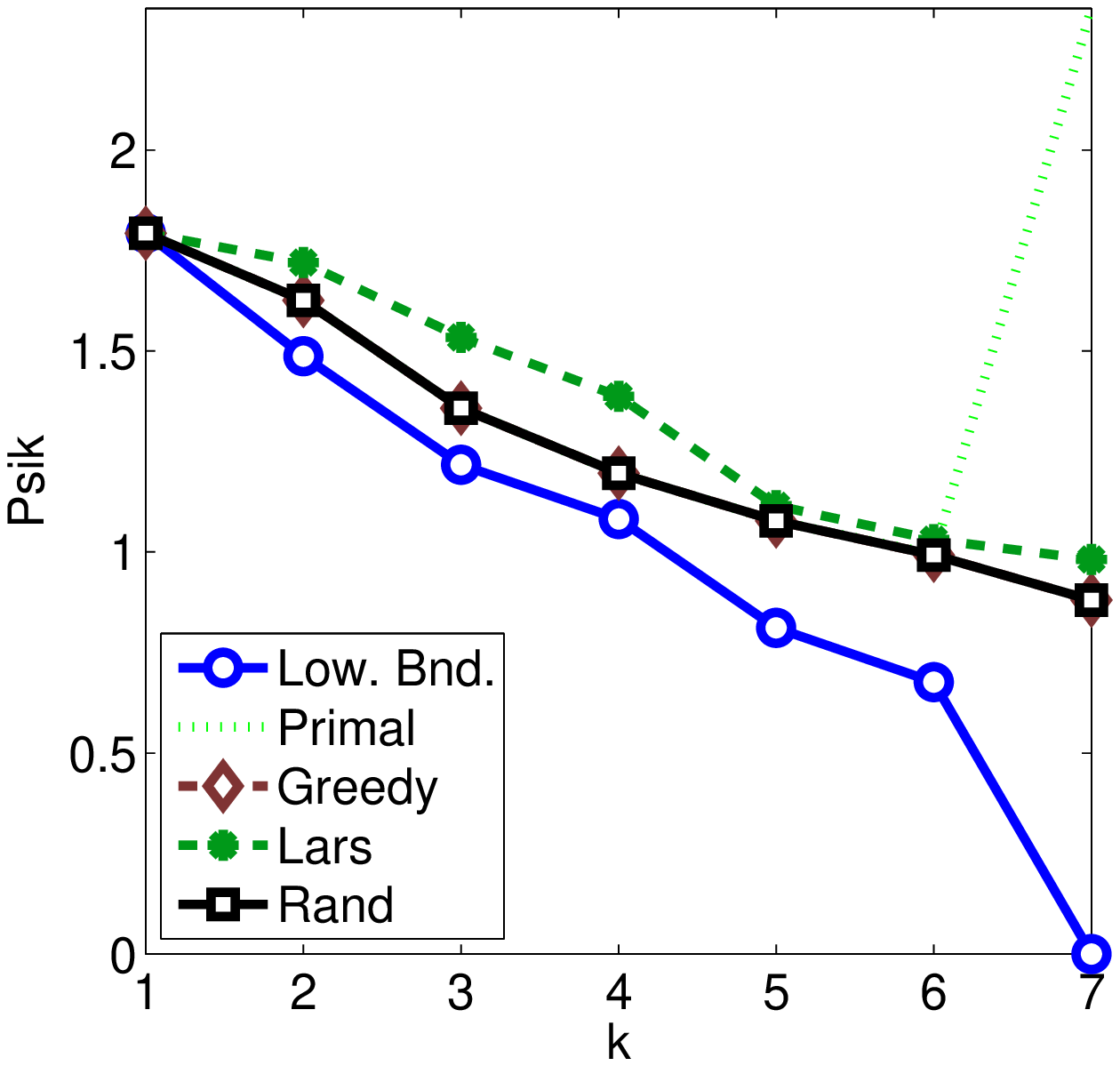} &
\psfrag{k}[t][b]{Cardinality $k$}
\psfrag{Psik}[b][t]{$\psi(k)$}
\includegraphics[width=.47\textwidth]{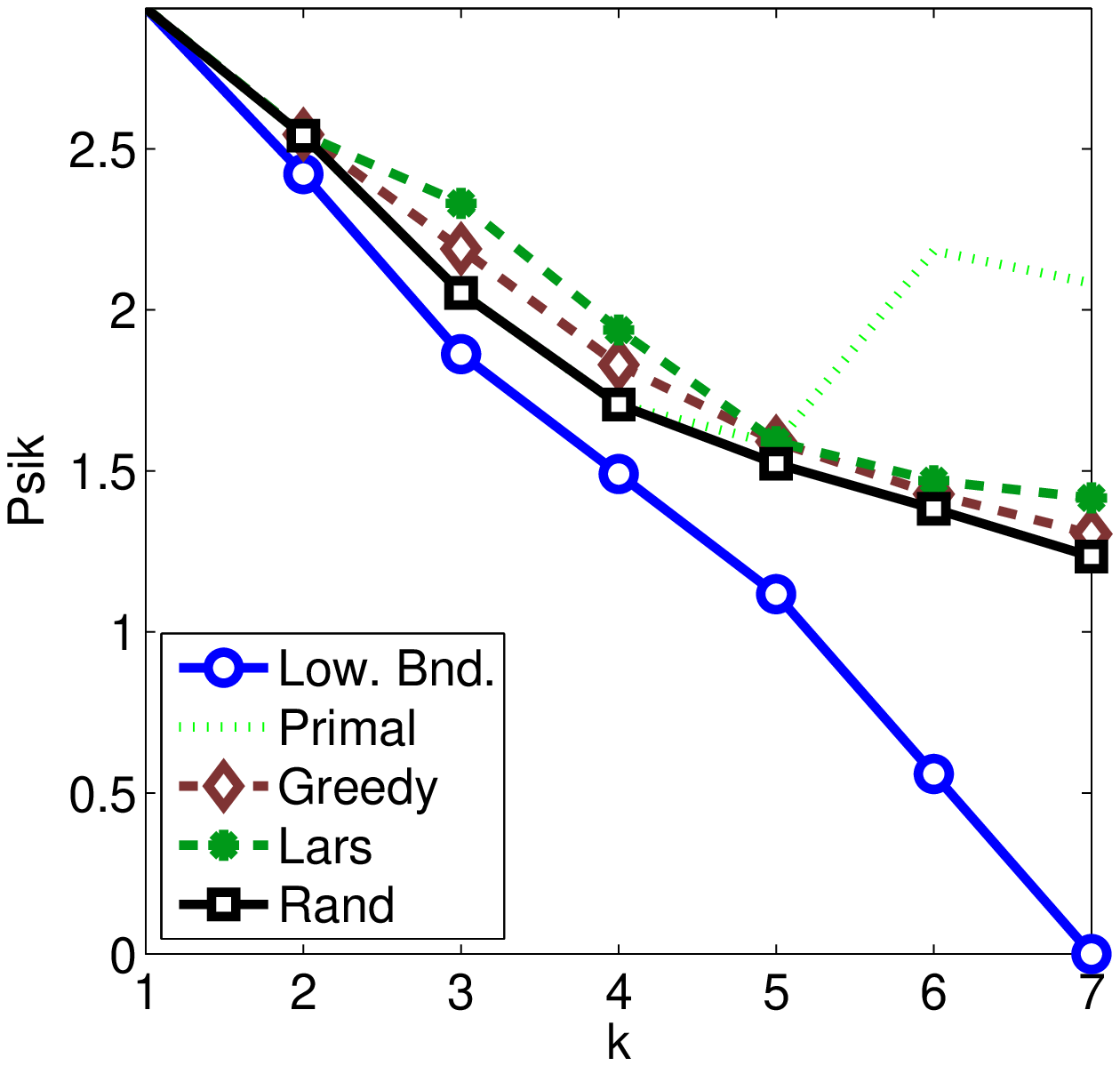}
\end{tabular}
\caption{Lower bounds and objective values in~\eqref{eq:subsel} for various algorithms, versus target cardinality. The true cardinality is 2 (left) and 4 (right). \label{fig:ilst}}
\end{center}
\end{figure}

More realistic data sets were generated with an image compression setting in mind. $X \in \reals^{n \times p}$ is now an overcomplete dictionary of Gabor wavelets, and $y$ is an image patch of size $r \times r$ obtained from an actual image. We set $r=10$ for all the experiments. We first solve this batch of problems (for $p=24$ and $n = 16$)  with the B\&B algorithm where the target cardinality is either 2, 3, or 4. We then compare the performance of the forward greedy algorithm and the enhanced randomization algorithm of Section~\ref{ss:rand}. Table \ref{t:gabor1} shows that the modified randomization algorithm finds the optimal solution in most cases.


\begin{table}[ht]
{\small
\begin{center}
\caption{Number of instanced solved by greedy and randomization algorithms on image data.}\label{t:gabor1} \vspace{1ex}
\begin{tabular}{|l|l|l||c||c|c||c|c|}\hline
 \multicolumn{3}{|c||}{Dimensions} & & \multicolumn{2}{|c||}{Greedy} &\multicolumn{2}{|c|}{ Randomization} \\ \hline
$p$   &  $ n$  & $k$ & No. instances &  No. solved &  Max. Rel. Gap & No. Solved & Max. Rel. Gap \\ \hline
24  &   16  &   2   &   10 & 9   & 0.22  & 9 &   0.90\\
24  &   16 &   3   &   10 & 8   &  0.70 & 9 & 0.16  \\
24  &   16 &   4   &   10 & 8   &  0.94 & 9 & 0.31  \\
  \hline
\end{tabular}\end{center}}
\end{table}

Most of our experiments so far were focused on finding {\em exact} solutions to small instances of problem~\eqref{eq:subsel}. We also tested the numerical complexity of our methods on larger problems for which we only sought good upper and lower bounds. Computing times for solving relaxation~\eqref{eq:relax} on increasingly large Gaussian random problems (generated as above) are reported in Table~\ref{tab:cpu}.

\begin{table}[h!]
{\small
\begin{center}
\caption{CPU time versus problem size.}\label{tab:cpu} \vspace{1ex}
\begin{tabular}{|r|c|}\hline
Problem size $p$ & CPU time\\
\hline
100 & 0 h 00 m 07 s\\
250 & 0 h 01 m 32 s\\
500 & 0 h 10 m 19 s\\
1000 & 1 h 22 m 59 s\\
\hline
\end{tabular}\end{center}}
\end{table}

\subsubsection*{Acknowledgments}
The last author would like to acknowledge partial support from NSF grants SES-0835550 (CDI), CMMI-0844795 (CAREER), CMMI-0968842, a Peek junior faculty fellowship, a Howard B. Wentz Jr. award and a gift from Google. 

\bibsep 1ex
\small{\bibliographystyle{plainnat}
\bibliography{MainPerso}}

\end{document}